\documentclass[a4paper,11pt]{amsart}
\usepackage{amssymb,amsmath,amsthm}
\usepackage{graphicx,color}
\usepackage{stackrel}
\usepackage{overpic}
\usepackage{hyperref}
\usepackage{float}

\newcommand{\R}{\mathbb{R}}

\newcommand{\eqskip}{ \vspace*{2mm}\\ }

\newtheorem{Lemma}{Lemma}
\newtheorem{Theorem}{Theorem}

\newtheorem{Proposition}{Proposition}

\theoremstyle{definition}
\newtheorem{Remark}{Remark}

\numberwithin{equation}{section}
\usepackage[normalem]{ulem}

\definecolor{DarkGreen}{rgb}{0,0.5,0.1}

\newcommand\soutD{\bgroup\markoverwith
{\textcolor{DarkGreen}{\rule[.5ex]{2pt}{1pt}}}\ULon}
\newcommand\soutP{\bgroup\markoverwith
{\textcolor{blue}{\rule[.5ex]{2pt}{1pt}}}\ULon}
\newcommand{\Hm}[1]{\leavevmode{\marginpar{\tiny%
$\hbox to 0mm{\hspace*{-0.5mm}$\leftarrow$\hss}%
\vcenter{\vrule depth 0.1mm height 0.1mm width \the\marginparwidth}%
\hbox to
0mm{\hss$\rightarrow$\hspace*{-0.5mm}}$\\\relax\raggedright #1}}}

\begin{document}

\title[Orbital stability]{On the existence and partial stability of standing waves for a nematic liquid crystal director field equations}
\author[Amorim, Casteras, and Dias]{ Paulo Amorim, Jean-Baptiste Casteras, and Jo\~{a}o-Paulo Dias}

\address{Instituto de Matem\'{a}tica, Universidade Federal do Rio de Janeiro, Av. Athos da Silveira Ramos 149, C.P. 68530, 21941-909 Rio de Janeiro, RJ- Brasil}
\email{paulo@im.ufrj.br}

\address{CMAFcCIO \& Departamento de Matem\'{a}tica, Faculdade de Ci\^{e}ncias da Universidade de Lisboa, Campo Grande, Edif\'{\i}cio C6,
1749-016 Lisboa, Portugal} \email{jeanbaptiste.casteras@gmail.com} \email{jpdias@fc.ul.pt }

\date{\today}

\begin{abstract}
  In this paper, following the studies in~\cite{ADM}, we consider some new aspects of the motion of the director field of a nematic liquid crystal submitted to a magnetic field and to a laser beam. In particular, we study the existence and partial orbital stability of special standing waves, in the spirit of ~\cite{CL} and ~\cite{HS et al} and we present some numerical simulations.
\bigskip

  \bigskip
\begin{itemize}
\item[\textbf{Keywords:}]
liquid crystal, nematic, magnetic field, laser beam, standing waves, stability.
\item[\textbf{MSC (2020):}]
76D03,76W05,78-10,35L53, 35C07.
\end{itemize}
\end{abstract}

\maketitle

\section{Introduction and main results}\label{Introduction}

 A great number of technological applications related to data display and non-linear optics, use thin films of nematic liquid cristals, cf. ~\cite{De Gennes} for the general theory  of nematic liquid cristals. In such devices the local direction of the optical axis of the liquid crystal is represented by a unit vector $\mathbf n(x,t)$,  called  the director, and may be modified by the application of an electric or magnetic field. The interaction of a light beam with the dynamics of the director $\mathbf n(x,t)$, under a magnetic field, helps to improve the device performance. 

In this paper we consider the model introduced in  ~\cite{ADM} to describe the motion of the director field of a nematic liquid crystal submitted to an external constant strong magnetic field  $\mathbf H$, with intensity $H \in \R$, and also to a laser beam, assuming some simplifications and approximations motivated by previous experiments and models (cf. ~\cite{Baqer et al}, ~\cite{Baqer et al 1},~\cite{Martins et al} and ~\cite{MI}, for magneto-optic experiments, and ~\cite{HS},~\cite{ZZ} for the simplified director field equation).\\
The system under consideration reads
\begin{equation} 
  \label{(1.1)}
  \begin{array}{ll}
    \left\{
    \begin{array}{l}
      i u_{t} + u_{xx}= -\rho u + a |u|^{2}u + H^{2} x^{2} u\eqskip
      \rho_{tt} =  (\sigma (v))_{x} - b\rho + |u|^{2},
    \end{array}
    \right.
    &   x\in\R,\; t\geq 0,
  \end{array}
\end{equation}
where $i$ is the imaginary unit, $u(x,t)$ is a complex valued function representing the wave function associated  to the laser beam under the presence of the magnetic field  $\mathbf H$ orthogonal to the director field, $ \rho \in \R $ measures the angle of the director field  with de $x$ axis,$v=\rho_{x}$, $a,H \in \R , b > 0$ are given constants,  with initial data
\begin{equation} 
\label{(1.2)}
  u(x,0)=u_{0}(x), \rho(x,0)= \rho_{0}(x), \rho_{t}(x,0)=\rho_{1}(x), x \in \R,
\end{equation}
and where the function $\sigma(v)$ is given by
\begin{equation}
\label{(1.3)} \sigma(v) = \alpha v + \lambda v^3, \quad  \lambda = \frac{2}{3} \gamma (\alpha-\beta),
\end{equation}
where $\alpha \geq \beta > 0$ are elastic constants of the liquid crystal, cf. ~\cite{HS}, and
\begin{equation} 
\label{(1.4)} \gamma = 4(\chi_a)^{-1} H^{-2} \beta > 0,
\end{equation}
where  $\chi_a > 0$ is the anisotropy of the magnetic susceptibility, cf. ~\cite{Martins et al}.\\
In the quasilinear case $\alpha > \beta, \alpha \simeq \beta$, the study of the existence of a weak global solution  to the Cauchy problem for the system \eqref{(1.1)} with the initial data \eqref{(1.2)}, in suitable spaces, has been developed in ~\cite{ADM}, by application of the compensated compactness method introduced in ~\cite{SS} to the regularised system with a physical viscosity and the vanishing viscosity method (cf.~also ~\cite{DFF} and ~\cite{DFO} for two examples of this technique applied to related systems of short waves-long waves).\\
In Section 2 we prove, in the general case ($\lambda \geq 0$), by application  of Theorem 6 in ~\cite{K}, a local in time existence and uniqueness theorem of a classical solution for the Cauchy problem ~\eqref{(1.1)},\eqref{(1.2)}.  For this purpose we need to introduce some functional spaces and point out several well known results:\\
Let be the linear operator defined in $L^2(\R)$ by
\begin{equation} 
\label{(1.5)} Au = u_{xx} - H^2x^2u ,\, u\in D(A), H \neq{0},
\end{equation}
where $D(A)= \big\{ u \in X_A | Au \in L^2 (\R) \big\}$ , with
\begin{equation} 
\label{(1.6)} X_A = \big\{ u \in H^1(\R) | xu \in L^2 (\R) \big\}.
\end{equation}
It can be proved, cf.~\cite{W}, that if $u\in X_1=\big\{ u  | xu, u_x \in L^2 (\R) \big\}$, then $u\in L^2(\R)$ with (denoting by $\|.\|_p$  the norm $\|.\|_{L^p(\R)}$)
\begin{equation} 
\label{(1.7)} \|u\|^2_2 \leq 2^{\frac{1}{2}}\|u_x\|_2 \|xu\|_2, \forall u \in X_1,
\end{equation}
and so $X_A=X_1$, and it is not difficult to prove that the injection of $X_A$ in $L^q(\R), 2\leq q < + \infty$, is compact (cf. ~\cite{HS et al}).\\
Moreover, it may be also proved, cf. ~\cite {C}, lemma 9.2.1, that $A$ is self-adjoint in $L^2(\R)$, $(Au,u) \leq 0, \forall u\in D(A)$, and (cf.~\cite{HO}),
\begin{equation} 
\label{(1.8)} D(A) = (-A+1)^{-1} L^2(\R) = \big\{ u \in H^2(\R) | x^2 u \in L^2 (\R) \big\}.
\end {equation}
We can now state the first result that will be proved in Section 2:
\begin {Theorem} \label{T1} Let $(u_0,\rho_0,\rho_1)\in D(A)\times H^3\times H^2$ and $\lambda \geq 0$. Then, there exists  $ T^* = T^*(u_0,\rho_0,\rho_1) > 0$ such that, for all  $T < T^*$, there exists  an unique solution $(u,\rho)$ to the Cauchy problem \eqref{(1.1)},\eqref{(1.2)} with
$u\in C ( [0,T] ; D(A) )\cap C^1( [0,T] ; L^2 )$ and $ \rho\in C ( [0,T] ; H^3 )\cap C^1( [0,T] ; H^2 ) \cap C^2 ( [0,T] ; H^1)$.
\end{Theorem}
As it is well known, in the quasilinear case the local solution, in general, blows-up in finite time. 

In Section 3, by obtaining the convenient estimates, we prove the following result in the semilinear case ($\alpha = \beta$):
\begin{Theorem} \label{T2} Let $(u_0,\rho_0,\rho_1)\in D(A)\times H^3\times H^2$ and $\lambda=0$. Then, there exists an unique global in time solution $(u,\rho)$ to the Cauchy problem \eqref{(1.1)},\eqref{(1.2)}, with $u\in C ( [0,+\infty) ; D(A) )\cap C^1( [0,+\infty) ; L^2 )$ and $ \rho\in C ( [0,+\infty) ; H^3 )\cap C^1( [0,+\infty) ; H^2 ) \cap C^2 ( [0,+\infty); H^1)$.
\end{Theorem}
In the special case of initial data with compact support, we will prove in Section 4  the following result:
\begin{Theorem} \label {T3} Assuming the hypothesis of Theorem ~\ref{T2}, consider the particular case where
\begin{equation} 
\label{(1.9)} \it{supp} \big\{u_0,\rho_0,\rho_1\big\} \subset D = ] -\theta, \theta [,\quad \theta > 0.
\end {equation}
Then, for each $ t > 0 $ and $ \varepsilon > 0 $  , there exists a $\delta = \delta ( t,\varepsilon,\|u_0\|_{H^1} ) > 0 $, such that
\begin{equation} 
\label{(1.10)} \int _{\R\setminus (D + B( 0, \delta))  } [ | u |^2 + |\rho|^2 + |\rho_x|^2 + |\rho_t|^2 ] (x,t) dx \leq \varepsilon,
\end{equation}
where $B( 0, \delta) = \big \{ x\in \R | |x| < \delta \big \} $.
\end{Theorem}

The proof of this result follows a technique introduced in ~\cite{Cor} in the case of the nonlinear Schr\"{o}dinger equation.\\
\indent In Section 5, which contains the main result in the paper, we study the existence and possible partial orbital stability of the standing waves for the system \eqref {(1.1)} with $a=-1$ (attractive case) and  $\lambda \geq 0 $.
These solutions are of the form
\begin{equation} 
\label{(1.11)} ( e^{i\mu t} u(x), \rho (x) ), \,  \mu \in \R,
\end{equation}
and the system ~\eqref{(1.1)} takes the aspect (we fix $\alpha=1$, without loss of generality):
\begin{equation} 
 \label{(1.12)}
  \begin{array}{ll}
    \left\{
    \begin{array}{l}
      u_{xx}  - H^{2} x^{2} u +  |u|^{2}u + \rho u = \mu u \eqskip
     - \rho_{xx}- \lambda (\rho_x^3)_x + b\rho = |u|^{2},
    \end{array}
    \right.
    &   x\in\R .
  \end{array}
\end{equation}
 We can rewrite this system as a scalar equation
\begin{equation} 
 \label{(1.14)}
  u_{xx}  - H^{2} x^{2} u +  |u|^{2}u + \rho (|u|^2) u = \mu u,
\end{equation}
where $\rho (f)$ is the solution to $ - \rho_{xx}- \lambda (\rho_x^3)_x + b\rho = f$. It is not difficult to prove that if $f \in L^2$, there exists a unique $\rho \in H^2$. This allows for instance to prove that $\rho (|u|^2) u^2 \in L^1$ provided that $u\in X_A$. Now, to find nontrivial solutions of this equation belonging to $D(A)$, the domain of the linear operator defined by~\eqref {(1.5)}, we will closely follow the technique introduced in ~\cite{HS et al} for the case of the Gross-Pitaevskii equation. More precisely, we consider the energy functional defined in $X_A$ by
(with $\int . dx = \int_\R . dx$):\\
\begin{equation} 
\begin{split}
\label{(1.15)} \mathcal{E} (u) = \frac {1}{2} \int |u_x|^2 dx + \frac {1}{2} H^2 \int x^2|u|^2 dx \\ - \frac {1}{4} \int |u|^4 dx 
- \frac {1}{2} \int \rho (|u|^2) |u|^2 dx, \, u\in X_A,
\end{split}
\end {equation}
 and we look to solve the following constrained minimization  problem for a prescribed $c>0$:
\begin{equation} 
\label{(1.16)} \mathcal {I}_c = \inf \big \{ \mathcal{E} (u), u\in X_A, \,  \text{real}, \int |u(x)|^2 dx = c^2 \big \}.
\end {equation}
We start by proving the following result which corresponds to Lemma~1.2 in ~\cite {HS et al}.
\begin{Theorem} \label{T4} We have: \\
i) The energy functional $\mathcal{E}$ is $C^1$ on \, $X_A$ real.\\
ii) The mapping $c \rightarrow \mathcal{I}_c $ is continuous.\\
iii) Any minimizing sequence of $\mathcal{I}_c$ is relatively compact in $X_A$ and so, if $\big\{u_n \big\}_{n\in\mathbb{N}} \subset  X_A$ is a corresponding minimizing sequence,then there exists $u\in X_A$ such that $\|u\|^2_2 = c^2$ and $\lim_ {n\rightarrow +\infty} u_n = u$ in $X_A$. Moreover $u(x)= u(|x|)$ is radial decreasing and satisfies ~\eqref{(1.14)} for a certain $\mu \in \R$.
\end{Theorem} 
To prove this result we follow the ideas in ~\cite{HS et al} and introduce the real space
$ \tilde{X}_A = \big\{w=(u,v) \in X_A\times X_A\big\}$, for real $u$ and $v$, with norm
\begin{equation}
\label{(1.17)} \|w\|^2_{\tilde{X}_A}= \|u\|^2_{X_A} + \|v\|^2_{X_A} , u, v \in X_A, 
\end{equation}
and observe that if $u =u_1 + i u_2$,with $u_1=\mathcal{R}e\, u,  u_2=\mathcal{I}m\, u$, the equation ~\eqref {(1.14)} can be written in the system form:
\begin{equation} 
  \label{(1.18)}
  \begin{array}{ll}
    \left\{
    \begin{array}{l}
       u_{1xx} -   H^{2} x^{2} u_1 + |u|^{2}u_1 + \rho(|u|^2)u_1= \mu u_1 \eqskip
       u_{2xx} -   H^{2} x^{2} u_2 + |u|^{2}u_2 + \rho (|u|^2) u_2= \mu u_2
        \end{array}
    \right.
    &   x\in\R,
  \end{array}
\end{equation}
with $w=(u_1,u_2)\in \tilde{X}_A, u_1=\mathcal{R}e\, u,  u_2=\mathcal{I}m\, u$.\\
In the new space $\tilde{X}_A$, the functional defined in ~\eqref{(1.15)} takes the form, for $w=(u,v)\in \tilde{X}_A, |w|^4=(|u_1|^2 + |u_2|^2)^2,$
\begin{equation}
\begin{split}
\label{1.19} \mathcal{\tilde E} (u) = \frac {1}{2} \int |w_x|^2 dx + \frac {1}{2} H^2 \int x^2|w|^2 dx \\ - \frac {1}{4} \int |w|^4 dx 
- \frac {1}{2} \int \rho (|w|^2) |w|^2 d\xi, \, w\in \tilde{X}_A,
\end{split}
\end{equation}
and, for all $c>0$,we introduce
\begin{equation}
\label{(1.20)}  \mathcal{\tilde I}_c= \inf \big \{ \mathcal{\tilde E} (w), w \in \tilde{X}_A ,  \int |w(x)|^2 dx = c^2 \big \},
\end{equation}
and the sets 
$$  \mathcal{ W}_c = \big \{ u\in X_A, \| u \|^2_2= c^2,\mathcal {I}_c = \mathcal{E} (u), u>0 \big \},$$
 $$\mathcal{ Z}_c = \big \{ w\in \tilde{X}_A, \| w \|^2_2= c^2,\mathcal {\tilde I}_c = \mathcal{\tilde E} (w) \big \}.$$
 Following ~\cite{CL} and ~\cite{HS et al}, we introduce the following definition:\\
 
\noindent \textbf{Definition}: The set $\mathcal{Z}_c$ is said to be stable if $\mathcal{Z}_c \neq \varnothing$ and
  for all $\varepsilon > 0$, there exists $\delta > 0$ such that, for all $w_0=({u_1}_0,{u_2}_0) \in \tilde{X}_A $, we have, for all $t\geq 0$,
 $$  \inf_{w \in  \mathcal{ Z}_c} \|w_0 - w\|_{\tilde{X}_A} <\delta\Longrightarrow \inf_{w \in  \mathcal{ Z}_c}\|\psi(.,t) - w\|_{\tilde{X}_A} <\varepsilon , $$
 where  $\psi(x,t) = (u_1(x,t),u_2(x,t))$ corresponds to the solution $ u(x,t) = u_1(x,t) + i u_2(x,t) $ of the first equation in the Cauchy problem \eqref{(1.1)},\eqref{(1.2)},  with initial data $ u_0(x) = {u_1}_0(x) + i {u_2}_0(x) $ and where $\rho(x,t)= \rho(|u(x,t)|^2)(x,t)$ satisfies
$$ - \rho_{xx}- \lambda (\rho_x^3)_x + b\rho = |u(.,t)|^2 .$$ 

\indent This corresponds to the hypothesis $\rho_{tt} \simeq 0$, cf. ~\cite{Baqer et al}, ~\cite{Baqer et al 1} and ~\cite{Martins et al} . The local existence and uniqueness in $X_A$ to the corresponding Cauchy problem for the  Schr\"{o}dinger equation is a consequence of Theorem $3.5.1$ in ~\cite{C}. It is easy to get the global existence of such solution $\psi (t)$ if their initial data is closed to $\mathcal{Z}_c$. Indeed, denote by $T$ the maximal time of existence and suppose that $\mathcal{Z}_c$ is stable at least up to the time $T$. So, using the stability at time $T$, we see that $\psi (T)$ is uniformly bounded in $\tilde{X}_A$. Therefore, we can apply the local existence result for initial data $\psi(T)$. This contradicts the maximality of $T$ and yields to the global existence.   \\
Proceeding as in the proof of Theorem \ref{T3} (see in particular \eqref{conrho1}), we can show that 
$$\|\rho (|\psi (t)|^2) - \rho (|w|^2) \|_{H^1}\leq C\|\psi (t) +w\|_{L^2} \|\psi (t) - w \|_{\tilde {X_A}}, $$
where $C$ is a constant not depending on $t$. So, if $w$ is stable, we derive, in the conditions of the definition,
 \begin{equation}
\label{(1.21)} \inf_{w \in  \mathcal{ Z}_c} \|\rho(.,t) - \varphi\|_{H^1} < c_1( \|u_0\|_2 +c)\varepsilon.
 \end{equation}  
 We point out that, if $w=(u_1,u_2)\in  \mathcal{ Z}_c$, then there exists a Lagrange multiplier $\mu\in\R$ such that $w$ satisfies ~\eqref{(1.18)}, that is
 $u = u_1 + i u_2$ satisfies ~\eqref{(1.14)}.\\
 \indent We will prove the following result which is a variant of Theorem 2.1 in ~\cite{HS et al}:
\begin{Theorem} \label{T5} The functional $\tilde{\mathcal{E}}$ is $C^1$ in  $\tilde{X}_A$  and  we have\\
i) For all $c>0, \mathcal{I}_c = \mathcal{\tilde I}_c , \mathcal{Z}_c \neq\varnothing$ and $\mathcal{Z}_c$ is stable.\\
ii) For all $w\in \mathcal{Z}_c, |w| \in \mathcal{W}_c$.\\
iii) $\mathcal{Z}_c = \big\{ e^{i\theta}u, \theta\in \R\big\}$, with $u$ real being a minimizer of \eqref{(1.16)}.
\end{Theorem}
The proof of this result is similar to the proof of Theorem 2.1 in ~\cite{HS et al}. 
We repeat some parts of the original proof for sake of completeness. Next, in Section 6, also following closely \cite{HS et al}, we prove a bifurcation result asserting in particular that all solutions of the minimisation problem \eqref{(1.16)} belongs to a bifurcation branch starting from the point $(\lambda_0 ,0)$ (in the plane $(\mu ,u)$) where $\lambda_0$ is the first eigenvalue of the operator $-\partial_{xx} +H^2 x^2$. 

\begin{Proposition}
\label{propbif}
The point $(\lambda_0 ,0)$ is a bifurcation point for \eqref{(1.14)} in the plane $(\mu ,u)$ where $-\mu \in \R^+$ and $u\in X_A$. The branch issued from this point is unbounded in the $\mu$ direction (it exists for all $-\mu>\lambda_0$). Moreover solutions to \eqref{(1.14)} belonging to this branch are in fact minimizers of problem \eqref{(1.16)}.
\end{Proposition}
As already mentioned, the proof of this proposition follows closely the one of \cite[Theorem 3.1]{HS et al}. An important ingredient which has also independent interest is the following uniqueness result.

\begin{Proposition}
\label{propuni}
There exists a unique radial positive solution to \eqref{(1.14)} such that $\lim_{r\rightarrow \infty} u(r)=0$.
\end{Proposition}
The proof of this proposition is strongly inspired by \cite{Hirose}.

\indent Finally, in Section 7 we present some numerical simulations illustrating the behaviour of the standing waves according to the intensity of the magnetic field $\mathbf H$, and also the limit as the Lagrange multiplier $-\mu$ approaches the bifurcation value $\lambda_0$.

\section{Local existence in the general case} \label{Local existence}

 In order to prove Theorem 1, let us introduce the Riemann invariants associated to the second equation in the system ~\eqref{(1.1)},
\begin{equation}
\label{ 2.1}l=w+\int_0^v\sqrt{\alpha+3\lambda\xi^2}d\xi\quad\textrm{and}\quad r=w-\int_0^v\sqrt{\alpha+3\lambda\xi^2}d\xi,
 \end{equation}
 where $w=\rho_t, v=\rho_x$.
 We derive
 $$ l-r=2\int_0^v\sqrt{\alpha+3\lambda\xi^2}d\xi$$ 
 $$=v\sqrt{\alpha+3\lambda v^2}+\frac {1}{\sqrt{3\lambda}}\,\textrm{arcsinh}(\sqrt{3\lambda} v),\quad w=\frac {l+r}2.$$
 Noticing that
 $$\displaystyle f(v)=v\sqrt{\alpha+3\lambda v^2}+\frac 1{\sqrt{3\lambda}}\,\textrm{arcsinh}(\sqrt{3\lambda} v)$$ 
 is one-to-one and smooth, we have $v=f^{-1}(l-r)=v(l,r)$ and, for classical solutions,
 the Cauchy problem ~\eqref{(1.1)},\eqref{(1.2)} is equivalent to the system
 \begin{equation}
\label{(2.2)}
\left\{
\begin{array}{lllll}
iu_t+u_{xx}- H^2x^2u=-\rho u+ a |u|^2u\\
\\
\rho_t=\frac 12 (l+r)\\
\\
l_t-\sqrt{\alpha+3\lambda v^2}l_x= - b \rho + |u|^2\\
\\
r_t+\sqrt{\alpha+3\lambda v^2}r_x= - b \rho + |u|^2\\

\end{array}
\right.
\end{equation}
with initial data (cf. ~\eqref{(1.5)},\eqref{(1.8)}),
\begin{equation}
\begin{split}
\label{(2.3)}
u(.,0)=u_0\in D(A)=\big\{ u \in H^2{\R} | x^2 u \in L^2 (\R) \big\} , \\ \rho(.,0)=\rho_0\in H^3(\R), l(.,0)=l_0\in H^2(\R),r(.,0)=r_0\in H^2(\R).
\end{split}
\end{equation}
In order to apply Kato's theorem (cf.~\cite[Thm. 6]{K}) to obtain the existence and uniqueness of a local in time strong solution, cf. Theorem~\ref{T1}, for the corresponding Cauchy problem, we need to pass to real spaces, introducing the variables
\begin{equation}
\label{(2.4)}u_1=\mathcal{R}e\, u,  u_2=\mathcal{I}m\, u.
\end{equation}
Now, we can pass to the proof of Theorem 1:\\
\indent With $ ({u_{1}}_{0},{u_{2}}_{0}) = (u_1(.,0),u_2(.,0))$, let
\begin{equation}
\label{(2.5)} U = (u_1,u_2,\rho,l,r), U_0 = ({u_{1}}_{0},{u_{2}}_{0},\rho_0,l_0,r_0),
\end{equation}
and
\begin{displaymath}
 \mathcal{A}(U)=\left[
\begin{array}{ccccccccccc}
0&A&0&0&0\\
-A&0&0&0&0\\
0&0&0&0&0\\
0&0&0&-\sqrt{\alpha+3\lambda v^2}\frac{\partial}{\partial x}&0\\
0&0&0&0&\sqrt{\alpha+3\lambda v^2}\frac{\partial}{\partial x}
\end{array}\right],
\end{displaymath}

\medskip

\begin{displaymath}
 g(t,U)=\left[
\begin{array}{cccccccccc}
-\rho u_2 + a (u_1^2+u_2^2)u_2\\
\rho u_1 - a (u_1^2+u_2^2)u_1\\
\frac 12(l+r)\\
-b\rho + |u|^2\\
-b\rho + |u|^2
\end{array}\right].
\end{displaymath}

The initial value problem \eqref{(2.2)}, \eqref{(2.3)} can be written in the form
\begin{equation}
\label{(2.6)}
\left\{
\begin{array}{ccccccc}
\displaystyle\frac{\partial}{\partial t}U+\mathcal{A}(U)U=g(t,U)\\
U(.,0)=U_0.
\end{array}\right.
\end{equation} 
Let us take
$$ U_0 = ({u_{1}}_{0},{u_{2}}_{0},\rho_0,l_0,r_0) \in Y= (D(A))^2\times(H^2(\R))^3$$
(the condition $\rho_0 \in H^3(\R)$ will be used later).
We now set $ X = (L^2(\R))^2\times(L^2(\R))^3$ and $S=((1-A)I)^2\times((1-\Delta)I)^3$,  which is an isomorphism $S:Y \rightarrow X$.
Furthermore, we denote by $W_R$ the open ball in $Y$ of radius $R$ centered at the origin and by $G(X,1,\omega)$ the set of linear operators $\Lambda\,:\,D(\Lambda)\subset X\to X$ such that:
\begin{itemize}
\item $-\Lambda$ generates a $C_0$-semigroup $\{e^{-t\Lambda}\}_{t\in\R_+}$;
\item for all $t\geq 0$, $\|e^{-t\Lambda}\|\leq e^{\omega t}$, where, for all $U\in W_R$, 
$$\omega=\frac 12 \sup_{x\in\R}\|\frac{\partial}{\partial x}a(\rho,l,r)\|\leq c(R), \quad c: [0,+\infty[\to [0,+\infty[ \textrm{ continuous, and }$$
\begin{displaymath}
 a(\rho,l,r)=\left[
\begin{array}{cccccccccc}
0&0&0\\
0&-\sqrt{\alpha+3\lambda v^2}&0\\
0&0&\sqrt{\alpha+3\lambda v^2}
\end{array}\right].
\end{displaymath}
\end{itemize}
By the properties of the operator $A$ (cf. Section 1) and following ~\cite[Section 12]{K}, we derive
$$\mathcal{A}\,: U= (u_1,u_2,\rho,l,r)\in W_R \to G(X,1,\omega),$$
and it is easy to see that $g$ verifies, for fixed $T>0$, $\|g(t,U(t))\|_Y\leq \theta_R$, $t\in[0,T]$, $U\in C([0,T];W_R)$.\\
For $(\rho,l,r)$ in a ball $\tilde{W}$ in $(H^2(\R))^3$, we set (see \cite[(12.6)]{K}), with $[.,.]$ denoting the commutator matrix operator,
$$B_0(\rho,l,r)=[(1-\Delta),a(\rho,l,r)](1-\Delta)^{-1}\in\mathcal{L}((L^2(\R))^3).$$
We now introduce the operator $B(U)\in\mathcal{L}(X)$, $U=(F_1,F_2,\rho,l,r) \in W_R$, by
\begin{displaymath}
 B(U)=\left[
\begin{array}{cccccccccc}
0&0&0&0&0\\
0&0&0&0&0\\
0&0\\
0&0&&B_0(\rho,l,r)\\
0&0
\end{array}\right].
\end{displaymath} 

In \cite[Section 12]{K}, Kato proved that for $(\rho,l,r)\in\tilde{W}$ we have
$$(1-\Delta)a(\rho,l,r)(1-\Delta)^{-1}=a(\rho,l,r)+B_0(\rho,l,r).$$
Hence, we easily derive 
 $$S\mathcal{A}(U)S^{-1}=\mathcal{A}(U)+B(U), U\in W_R.$$\\
 Now, it is easy to see that conditions (7.1)--(7.7) in Section~7 of ~\cite{K} are satisfied and so we can apply Theorem~6 in ~\cite{K} and we obtain the result 
 stated in Theorem 1, with $ \rho\in C ( [0,T] ; H^2 )\cap C^1( [0,T] ; H^1 ) \cap C^2 ( [0,T] ; L^2)$.\\
 To obtain the requested regularity for $\rho$ it is enough to remark that, since $\rho_x=v, \rho_t = w, \rho_0 \in H^3, v_0={\rho_0}_x\in H^2, w_0 =\rho_1 \in H^2$,
 we deduce $\rho_x =v \in C ( [0,T] ; H^2 ), \rho_t=w \in C ( [0,T] ; H^2 )$, and this achieves the proof of Theorem 1.

\section{Global existence in the semilinear case} \label{Global existence}

Now, we consider the semilinear case, that is when $ \alpha = \beta $ and so $\lambda = 0$.\\
Hence we pass to the proof of Theorem 2. For the local in time unique solution $(u,\rho)$ defined in the interval $[0,T^*[, T > 0$, to the Cauchy problem \eqref{(1.1)},\eqref{(1.2)}, obtained in Theorem 1, we easily deduce the following conservation laws (cf. ~\cite {ADM}) in the case $\lambda \geq 0, \alpha >0 $:\\ 
  \begin{equation} 
  \label{(3.1)}    \int |u(x,t)|^{2}\;dx = \int |u_{0}(x)|^{2}\;dx,  t \in [0,T^*[.
  \end{equation}
 \begin{equation}
 \begin{split}
 \label{(3.2)} E(t) = \frac{1}{2} \int (\rho_{t}(x,t))^{2}\;dx + \frac{\alpha}{2} \int (\rho_{x} (x,t))^{2}\;dx
+\frac{\lambda}{4} \int (\rho_{x}(x,t))^{4}\;dx\\ 
+\frac{b}{2} \int (\rho(x,t))^{2} \;dx -  \int \rho(x,t) |u(x,t)|^{2} \;dx +\int |u_{x}(x,t)|^{2} \;dx\\ + \frac{a}{2} \int |u(x,t)|^{4}\;dx +H^{2} \int x^{2}|u(x,t)|^{2}\;dx = E(0), t \in [0,T^*[.
 \end{split}
 \end{equation}
 Applying the Gagliardo-Nirenberg inequality to the term $|\frac{a}{2} \int |u(x,t)|^{4}\;dx|$ and since $b>0$ we easily derive 
 (cf.~\cite{ADM}), for $t\in[0,T^*[$,
\begin{equation}
\begin{split}
\label{(3.3)} \int (\rho_{t}(x,t))^{2}\;dx +  \int (\rho_{x} (x,t))^{2}\;dx +\lambda \int (\rho_{x}(x,t))^{4}\;dx\\ 
+ \int (\rho(x,t))^{2} \;dx  +\int |u_{x}(x,t)|^{2} \;dx + H^{2} \int x^{2}|u(x,t)|^{2}\;dx \leq c_1.
\end{split}
\end{equation}
We continue with the proof of Theorem 2, in the semilinear case, that is  $\lambda=0$. We have,  for $t\in[0,T^*[,$ 
\begin{equation}
\label{(3.4)} \|\rho(t)\|_2 \leq \|\rho_0\|_2 + \int_{0}^{t} \|\rho_t(\tau)\|d\tau \leq c_2(1+ t).
\end{equation}
Next we estimate $\|Au(t)\|_2,\|\rho_{xt}(t)\|_2$ and $\|\rho_{xx}\|_2$. For $\lambda = 0$, the system ~\eqref{(2.2)} reads
 \begin{equation}
\label{(3.5)}
\left\{
\begin{array}{ll}
iu_t+u_{xx}- H^2x^2u=-\rho u+ a |u|^2u\\
\rho_t=\frac 12 (l+r)\\
l_t-\sqrt{\alpha}\, l_x= - b \rho + |u|^2\\
r_t+\sqrt{\alpha} \, r_x= - b \rho + |u|^2\\
\end{array}
\right.
\end{equation}
with initial data ~\eqref{(2.3)}. To simplify, we assume $\alpha=\beta=b=1$.\\
Recall that we have, since $\lambda=0$,
\begin{equation}
\label{(3.6)}
\left\{
\begin{array}{l}
l= w+v= \rho_{t}  + \rho_{x}\\
\\
r=w-v= \rho_{t}  - \rho_{x}.\\
\end{array}
\right.
\end{equation}
From ~\eqref{(3.5)}, we derive
$$ r_{tx}r_x + r_{xx}r_x = -\rho_x r_x + 2 \mathcal{R}e(\bar{u}u_x)r_x,$$
and so
$$\frac{1}{2}\frac{d}{d t}\int (r_x)^2 dx \leq \frac{1}{2}\int[(\rho_x)^2 + (r_x)^2] dx + c_3\int(r_x)^2 dx + c_3,$$
and a similar estimate for $l_x.$ We deduce, with $c_4(t)$ being a positive, increasing  and continuous function,
\begin{equation}
\label{(3.7)} \|r_x(t)\|^2_2 + \|l_x(t)\|^2_2 \leq c_4(t), t\in [0,T^*[.
\end{equation}
Moreover, we derive from ~\eqref{(3.5)}, formally,
$$ \mathcal{R}e (u_{tt}\bar{u}_t ) + \mathcal{I}m [ (u_{xxt} - H^2x^2u_t)\bar{u}_t] = a \mathcal{I}m[(|u|^2u)_t\bar{u}_t],$$
$$\frac{1}{2} \frac{d}{dt} \int |u_t|^2dx - \mathcal{I}m\int u_{xt}\bar{u}_{xt}dx = 2a\mathcal{I}m\int \mathcal{R}e( u\bar{u}_{t}) u\bar{u}_{t}dx$$ 
$\leq c_5\int |u_t|^2 dx$, and hence
\begin{equation}
\label {(3.8)}\int |u_t|^2 dx \leq c_6(t) , t\in [0,T^*[.
\end{equation}
We deduce from ~\eqref{(3.5)},
\begin{equation}
\label {(3.9)} \|Au(t) \|_2 \leq c_7(t), t\in [0,T^*[.
\end{equation}
We have by ~\eqref{(3.5)},
$$ r_{txx} r_{xx} + r_{txx} r_{xx} = -\rho_{xx} r_{xx}  + 2\frac{d}{d x}[\mathcal{R}e(u\bar{u}_x)] r_{xx}$$
and so, formally,
\begin{equation}
\begin{split}
\label {(3.10)}\frac{1}{2}\frac{d}{d t}\int (r_{xx})^2 dx \leq \frac{1}{2}\int[(\rho_{xx})^2 + (r_{xx})^2] dx \\ +2 \int ( |u| |u_{xx}| + |u_x|^2) |r_{xx}|dx\\
 \leq \frac{1}{2}\int[(\rho_{xx})^2 + (r_{xx})^2] dx + c_8(t)\int|r_{xx}|^2 dx,
\end{split}
\end{equation}
by ~\eqref{(3.9)} and ~\eqref{(1.8)}.
But, by ~\eqref{(3.6)}, we derive
$$ \rho_{xx} = \frac{1}{2}(l_x - r_x),$$
and so, by ~\eqref{(3.7)} and ~\eqref{(3.10)}, we deduce
\begin{equation}
\label {(3.11)} \frac{d}{dt} \int (r_{xx})^2 dx \leq c_9(t) \int (r_{xx})^2 dx +c_{10}(t)
\end{equation}
and similarly
\begin{equation}
\label {(3.12)} \frac{d}{dt} \int (l_{xx})^2 dx \leq c_9(t) \int (l_{xx})^2 dx +c_{10}(t).
\end{equation}
We conclude that
\begin{equation}
\label {(3.13)} \|r_{xx}\|^2_2+  \|l_{xx}\|^2_2 \leq c_{11}(t), t \in [0,T^*[,
\end{equation}
with $c_{11}(t)$ being a positive, increasing and continuous function of $t \geq {0}$.
This achieves the proof of Theorem 2 (the operations that we made formally can be easily justified by a convenient smoothing procedure).

\section{Special case of initial data with compact support} \label{Special case}

 We assume the hypothesis of Theorem 2, that is is we consider the semilinear case ($\lambda = 0$) and, without loss of generality, we take 
 $\alpha = \beta = b = |a| = 1$. We also assume that the initial data verifies ~\eqref{(1.9)} for a certain $ d > 0 $. Following ~\cite[Section 2]{Cor}, if we take 
 $ \phi \in W^{1,\infty} (\R), $ real valued, and $u$ is the solution of the  Schr\"{o}dinger equation in ~\eqref{(1.1)}, we easily obtain
 $$ \mathcal{R}e \int \phi^2 u_t\bar{u}dx  + \mathcal{I}m \int \phi^2 u_{xx}\bar{u}dx = 0. $$ 
 We derive
 \begin{equation}
\label {(4.1)} \|\phi u(t)\|_2 \leq \|\phi u_0\|_2 + c_0 t \|\phi_x\|_\infty,\quad t\geq 0, 
\end{equation}
where  
$$ c_0 = 2 \sup_{t \geq 0}  \|u_x(t)\|_2.$$
Moreover, from the wave equation in ~\eqref{(1.2)} with $\lambda = 0$, we deduce for $t \geq 0$,
$$ \phi^2 \rho_{tt}\rho_t - \phi^2 \phi_{xx}\rho_t = -  \phi^2 \rho \rho_t +\phi^2 \rho_t |u|^2,$$
\begin{equation}
\begin{split}
\label {(4.2)} \frac{d}{dt} \int (\phi \rho_t)^2 dx + \frac{d}{dt} \int (\phi \rho_x)^2 dx + \frac{d}{dt} \int (\phi \rho)^2 dx\\
= 2 \int \phi^2 \rho_t |u|^2 dx \leq 2 \|\phi \rho_t\|_2 \|\phi u\|_2 \|u_0\|_2.
\end{split}
\end{equation}
We assume
\begin{equation}
\label {(4.3)}  0 \leq \phi \leq 1.
\end{equation}
 We have, by the Gagliardo-Nirenberg inequality and ~\eqref{(4.1)},
\begin{equation}
\begin{aligned}
\label {(4.4)} \|\phi u\|_\infty \leq \| \phi u\|_2^\frac{1}{2} \|(\phi u)_x\|_2^\frac{1}{2} \\
\leq (\| \phi u_0\|_2 + c_0 t \| \phi_x\|_\infty))^\frac{1}{2} ( \|\phi_x\|_\infty \|u_0\|_2+ \frac{c_0}{2})^\frac{1}{2}\\
= g_0 (t).
\end{aligned}
\end{equation}
Now, with
\begin {equation}
\label{(4.5)} g_1(t) = g_0 (t) \| u_0 \|_2,
\end {equation}
we deduce, from ~\eqref{(4.2)},~\eqref{(4.4)} and with
\begin {equation}
\label{(4.6)} f_1(t) = \int (\phi \rho_t)^2 dx +  \int (\phi \rho_x)^2 dx +  \int (\phi \rho)^2 dx,
\end{equation}
$$f_1^\frac{1}{2}(t) \leq f_1^\frac{1}{2}(0)  + 2\int_ {0}^{t} g_1(\tau)f_1^\frac{1}{2} (\tau) d\tau, $$
\begin {equation}
\label{(4.7)}f_1^\frac{1}{2} (t) \leq f_1^\frac{1}{2} (0)  + \int_ {0}^{t} g_1(\tau) d\tau \leq f_1^\frac{1}{2}(0) + t \|u_0\|_2 \, g_0(t), t\geq 0.
\end{equation}
Hence, if we define
\begin {equation}
\label{(4.8)} f(t) = f_1(t) + \|\phi u(t)\|_2^2, \quad t\geq 0,
\end {equation}
we derive, by \eqref{(4.7)}, \eqref{(4.8)} and  \eqref{(4.1)},
\begin {equation}
\begin{split}
\label{(4.9)} f^\frac{1}{2} (t) \leq f_1^\frac{1}{2} (t) + \|\phi u(t)\|_2 \leq f_1^\frac{1}{2}(0) + t \|u_0\|_2 g_0(t) + \| \phi u_0\|_2 +c_0 t\|\phi_x\|_\infty \\
\leq f^\frac{1}{2} (0) + t \|u_0\|_2(\|\phi u_0\|_2 + c_0 t \| \phi_x\|_\infty))^\frac{1}{2} ( \|\phi_x\|_\infty \|u_0\|_2+ \frac{c_0}{2})^\frac{1}{2} \\
+ \| \phi  u_0\|_2 +c_0 t\|\phi_x\|_\infty.
\end{split}
\end{equation}
Now, we fix $ t > 0 $ and $ \varepsilon > 0 $  and assume that the initial data verify ~\eqref {(1.9)}. We introduce the set 
$ C = \R\setminus (D + B( 0, \delta))$, $\delta$ to be chosen,  and the function $ \phi \in W^{1,\infty} (\R), $ real valued, verifying ~\eqref {(4.3)},  
$\phi= 0$ in $D$, $\phi=1$ in $C$ and $ \| \phi_x\|_\infty =\frac {1}{ \delta } $. We have $f(0)=0$, $\phi u_0 = 0$, and so, by ~\eqref{(4.9)}, we easily obtain
\begin {equation}
\label{(4.10)} f(t) \leq 2c_0\|u_0\|_2^3 \frac{t^3}{{\delta}^2} + c_0^2\|u_0\|_2^2\frac{t^3}{\delta} + 2c_0^2\frac{t^2}{{\delta}^2}, 
\end{equation}
and now we can choose  $\delta$ such that ~\eqref {(1.10)} is satisfied and the Theorem 3 is proved.

\section{Existence and partial stability of standing waves} \label{Orbital stability}

 We will consider the system  ~\eqref{(1.1)} in the attractive case $a =-1$ and without loss of generality we assume that $\alpha=1$. We want to study the existence and behaviour of standing waves of the system ~\eqref {(1.1)}, that is solutions of the form ~\eqref {(1.11)}. As we have seen in the introduction, we can rewrite this system as a scalar equation \eqref{(1.14)}. Following the technique introduced in ~\cite{HS et al} fort the Gross-Pitaevski equation, we consider the energy functional defined in $X_A$ by ~\eqref {(1.15)}. Recall that $ X_A   \subset L^q (\R) , 2 \leq q < + \infty$, with compact injection, and the norm in $X_A$ is equivalent to the norm
\begin{equation}
\label {(5.1)} \| u \|_{X_A}^2 = \int |u_x|^2 dx  + H^2 \int x^2 |u|^2 dx, H \neq 0, u\in {X_A}.
\end{equation}
We now pass to the proof of Theorem 4, which is a variant of Lemma 1.2 in ~\cite{HS et al}, whose proof we closely follow.
Let $\big\{u_n\big\}$ be a minimizing sequence of $\mathcal{E}$ defined by ~\eqref{(1.15)} in 
$X_A$ (real), that is
$$ u_n \in X_A , \|u_n\|^2 = c^2, \lim_ {n\rightarrow +\infty} \mathcal{E} (u_n) = \mathcal{I}_c \, ,$$
defined by ~\eqref{(1.16)}. Multiplying the equation satisfied by $\rho (|u|^2)$ by $u$ and integrating by parts, we find
$$\int [\rho_x^2 + \lambda \rho_x^4 + b \rho^2 ] dx= \int |u|^2 \rho dx. $$
Using Young's inequality, we get, for a constant $C$ depending on $b$ (we allow this constant to change from line to line),
$$  \int |u|^2 \rho dx \leq \frac{b}{2}\int \rho^2 dx + C_b \int |u|^4 dx. $$
So using the two previous lines, we get that
\begin{equation}
\label{estrho}
\int \rho^2 dx \leq C_b \int |u|^4 dx.  
\end{equation}
From H\" older's inequality, we obtain that, for some constant $\tilde{C}>0$,
\begin{equation}
\label{(5.2)}\int \rho |u|^2 dx \leq (\int \rho^2 dx )^{\frac{1}{2}} (\int |u|^4 dx)^{\frac{1}{2}}\leq \tilde{C}  \int |u|^4 dx.  
\end{equation}
and, by Gagliardo-Nirenberg inequality,
\begin{equation}
\label{(5.3)} \|u\|^4 \leq C \|u_x\|_2 \|u\|_2^3 , u\in H^1(\R).
\end{equation}
Hence, reasoning as in ~\cite{HS et al}, \eqref{(1.1)} in lemma 1.2, we derive, for each $\varepsilon > 0$ and $x\in X_A$, such that $\|u\|_2^2 = c^2$,
\begin{equation}
\label{(5.4)} \|u\|_4^4 \leq \frac{\varepsilon^2}{2} \|u_x\|_2^2 + \frac{1}{2\varepsilon^2} c^6,
\end{equation}
and so, for $u\in X_A$ such that $\|u\|_2^2 = c^2$, we deduce
\begin{equation}
\label{(5.5)} \mathcal{E} \geq (\frac{1}{2} -\frac{\varepsilon^2}{2} (\frac{1}{4}+\frac{\tilde{C}}{2} ))\|u_x\|^2 - \frac{1}{2\varepsilon^2 }(\frac{1}{4}+\frac{\tilde{C}}{2} ) c^6 - \frac{1}{2} H^2\int x^2|u|^2 dx,
\end{equation}
and we can choose $\varepsilon$ such that $1- \epsilon^2  (\frac{1}{4}+\frac{\tilde{C}}{2} ) > 0$.\\
Hence, the minimizing sequence is bounded in $X_A$ and there exists a subsequence $\big{\{u_n}\big\}$ such that $u_n \rightharpoonup u$ in $X_A$ (weakly).
Recalling that the injection of $X_A$ in $L^4(\R)$ is compact, we derive
\begin{equation}
\label{(5.6)} u_n \rightarrow u\, \textit {in} \, L^4(\R).
\end{equation}
Moreover, by the lower semi-continuity, we deduce
\begin{equation}
\label{(5.7)} \int (|u_x|^2 + H^2x^2|u|^2) dx \leq \liminf_{n \to \infty}\int (|{u_n}_x|^2 + H^2x^2|u_n|^2) dx.
\end{equation}
On the other hand, we have, setting $f:= \rho (|u|^2) - \rho (|u_n|^2):=\rho - \rho_n$, 
$$- f_{xx} - \lambda (\rho_x^3 -(\rho_n)_x^3 )_x +b f = |u|^2 - |u_n|^2 .$$
Noticing using Young's inequality that 
$$-\int (\rho - \rho_n)  (\rho_x^3 -(\rho_n)_x^3 )_x dx =\int (\rho_x^4 +(\rho_n)_x^4 - \rho_x^3 (\rho_n)_x - (\rho_n)_x^3 \rho_x ) dx \geq 0.  $$
So proceeding as in \eqref{estrho}, we can show that
\begin{equation}
\label{conrho1}
\|f\|_{L^2}^2\leq C \| |u|^2- |u_n|^2\|_{L^2}^2 .
\end{equation}
Using this last estimate, we deduce that
\begin{align*}
&|\int \rho (|u|^2) |u|^2 - \int \rho (|u_n|^2) |u_n|^2  | \\
& \leq |\int \rho (|u|^2) (|u|^2 - |u_n|^2) dx| + |\int (\rho (|u|^2) - \rho (|u_n|^2)) |u_n|^2 dx |\\
&\leq  \| \rho (|u|^2)\|_{L^2}  \||u_n|^2 - |u|^2 \|_{L^4}^2+ \|u_n\|_{L^4}^2 \|\rho (|u|^2) - \rho (|u_n|^2) \|_{L^2}\rightarrow 0,
\end{align*}

since $ |u_n|^2 \rightarrow |u|^2$ in $L^2(\R)$.\\
Hence, $u$ is a minimizer of  ~\eqref {(1.15)}, that is
$$ u\in X_A , \|u\|_2^2 = c^2, \mathcal{E}  = \mathcal{I}_c .$$
We conclude that $ \mathcal{E}(u_n) \rightarrow\mathcal{E}(u)$ and  so
\begin{equation}
\label{(5.8)} \int |{u_n}_x|^2 dx+ H^2\int x^2 |u_n|^2 dx \rightarrow  \int |u_x|^2 dx+ H^2\int x^2 |u|^2 dx.
\end{equation}
We derive that $u\in X_A$ and, denoting  by $ u^ {\star}$  the Schwarz rearrangement of the real function $u$, (cf. ~\cite{LL} for the definition and general properties),  we know that
$$ \|u^ {\star}_x\|^2 \leq \|u_x\|^2, \|u^ {\star}\|^4 = \|u\|^4.$$
The Polya-Szego inequality asserts that, for any $f\in W^{1,p}$ with $p\in [1,\infty]$,
$$\|\nabla f\|_{L^p} \geq \|\nabla f^\star \|_{L^p}.$$
Moreover, by ~\cite{HS et al}, we have
\begin{equation}
\label{(5.9)} \int x^2|u^ {\star}|^2 dx< \int x^2|u|^2 dx, \textit{unless} \,\, u= u^ {\star}.
\end{equation}

By \cite[Theorem 6.3]{H}, we know that
$$\int G (v(x)) dx \leq \int G (v^\star (x)) dx ,$$
provided that $G(t)= \int_0^t g (s) ds$ and $g: \R_+ \rightarrow \R_+$ is such that
$$|g(s)|\leq K (s+s^l),$$
where $K>0$, $l>1$ and $s\geq 0$. We want to apply this result for $G(s)=\rho (s^2) s^2$. So $g(s)= (\rho (s^2))_s s^2 +2s \rho (s^2) $. Observe that $(\rho (s^2))_s :=f $ is the solution to $- f_{xx} - 3\lambda ((\rho (s^2))_x^2 f_x )_x +bf =2s$. Using the maximum principle, we can show that $g: \R_+ \rightarrow \R_+$. On the other hand, by standard elliptic regularity theory, we have that $|\rho (s^2)|, |\rho (s^2)_s|\leq C(s+s^2 )$, for any $s>0$. So, \cite[Theorem 6.3]{H} yields that
$$\int \rho (|u|^2) |u|^2 dx \leq \int \rho (|u^\star|^2) |u^\star|^2 dx. $$
Combining all the previous inequalities, we see that
$\mathcal{E} (u^{\star}) <  \mathcal{E} (u)$ unless $u= u^{\star}$ a.e. and this proves that the minimizers of ~\eqref {(1.15)} are non-negative and radial decreasing. This completes the proof of Theorem 4.\\

We now pass to the proof of  Theorem 5, which follows the lines of the proof of Theorem 2.1 in ~\cite{HS et al}. For sake of completeness we repeat some 
parts of the proof to make it easier to follow.\\
We recall that, cf. ~\cite{CL}, to prove the orbital stability it is enough to prove that $\mathcal{Z} \neq \varnothing$ and that any sequence $\big\{w_n=(u_n,v_n)\big\} \subset \tilde{X}_A$ such that $\|w_n\|_2^2 \rightarrow c^2$ and $\tilde{E}(w_n) \rightarrow \tilde{\mathcal I}_c$, is relatively compact in $\tilde{\mathcal E}$. 
By the computations in the proof of Theorem 4, we have that the sequence $\big\{w_n\big\}$ is bounded in $\tilde{X}_A$ and so we can assume that there exists a subsequence, still denoted by $\big\{w_n\big\}$ and $w=(u,v) \in \tilde{X}_A$ such that $w_n \rightharpoonup w$ weakly in
 $\tilde{X}_A$, that is $u_n \rightharpoonup u, v_n \rightharpoonup v$ in $X_A$. Hence, there exists a subsequence, still denoted by $\big\{w_n)\big\}$, such that there exists
 \begin{equation}
\label{(5.11)} \lim _{n \to \infty} \int ( |{u_n}_x|^2 + |{v_n}_x|^2 ) dx .
\end{equation}
Now, we introduce $\varrho_n = |w_n| = ( u_n^2 + v_n^2 )^\frac{1}{2} $, which belongs to ${X}_A$.
Following the proof of~\cite[Theorem 2.1]{HS et al}, we have\\
 ${\varrho_n}_x = \frac {u_n {u_n}_x + v_n {v_n}_x} {(u_n^2 + v_n^2)^\frac{1}{2} }$, if $u_n^2 + v_n^2 > 0$, and ${\varrho_n}_x = 0$, otherwise.\\
 We deduce
 \begin{equation}
 \begin{aligned}
 \label{(5.12)}
 \tilde {\mathcal{E}}(w_n)-\mathcal {E} (\varrho_n) = \frac{1}{2} \int _{u_n^2 + v_n^2 > 0} \Big(\frac {u_n {u_n}_x - v_n {v_n}_x}{u_n^2 + v_n^2}\Big)^2 dx\\
- \frac{1}{4} \int {|(|u_n|^2+ |v_n|^2)|^2}d\xi +  \frac{1}{4}\int {|\varrho_n|^4 d\xi}\\
= \frac{1}{2} \int _{u_n^2 + v_n^2 > 0} \Big(\frac {u_n {u_n}_x - v_n {v_n}_x}{u_n^2 + v_n^2}\Big)^2 dx.
\end{aligned}
\end{equation}
Hence, we derive as in~\cite[Theorem 2.1]{HS et al},
\begin{equation}
\label{(5.13)} {\tilde{\mathcal{I}}}_c = \lim_{n \to \infty} \mathcal{\tilde E}(w_n) \geq \limsup_{n \to \infty} \mathcal{E} (\varrho_n)
\end{equation}
and
\begin{equation}
\label{(5.14)} \lim_{n \to \infty} \|\varrho_n\|_2^2 =  \lim_{n \to \infty}  \|w_n\|_2^2 = c^2.
\end{equation}
Applying Theorem 4 and with $c_n = \|\varrho_n\|_2$, we obtain
\begin{equation}
\label{(5.15)} \liminf _{n \to \infty} \mathcal{E} (\varrho_n) \geq  \liminf_{n \to \infty}  \mathcal{I}_{c_n} \geq \mathcal{I}_c \geq \mathcal{\tilde{I} }_c.
\end{equation}
Hence, by ~\eqref{(5.13)}  and ~\eqref{(5.15)}, we derive
\begin{equation}
\label{(5.16)} \lim_{n \to \infty} \mathcal{E} (\varrho_n)=\lim_{n \to \infty} \mathcal{\tilde {E}}(w_n) = \mathcal{I}_c = \mathcal{\tilde {I}}_c \, ,
\end{equation}
and so, by ~\eqref{(5.12)} and ~\eqref{(5.16)}, we get
\begin{equation}
\label{(5.17)} \lim_{n \to \infty} \int  |{u_n}_x|^2 + |{v_n}_x|^2 - | \partial_x \big(( u_n^2+ v_n^2 )^{\frac{1}{2}}\big) |^2  dx = 0.
\end{equation}
We can rewrite this last line as
\begin{equation}
\label{(5.18)} \lim_{n \to \infty} \int ( |{u_n}_x|^2 + |{v_n}_x|)^2 dx = \lim_{n \to \infty} \int |{\varrho_n}_x|^2 dx.
\end{equation}
Now, by ~\eqref{(5.14)}, \eqref{(5.16)} and iii) in Theorem 4, we conclude that  there exists $\varrho\in X_A$ such that $\varrho_n \rightarrow \varrho$ in $X_A$ and
$\| \varrho \|_2^2 = c^2 , \mathcal{E} (\varrho) = \mathcal{I}_c.$ Moreover $\varrho \in H^2(\R) \subset C^1(\R)$  is a solution of ~\eqref{(1.14)} and $\varrho > 0$.
We prove that $\varrho =(u^2 + v^2)^{\frac{1}{2}} $ just as in the proof of Theo. 2.1 in ~\cite[p.279]{HS et al}.\\
\noindent Finally, we prove that $\|{w_n}_x\|_2^2\rightarrow \|w_x\|_2^2 $:\\
By applying ~\eqref{(5.18)} we have
$\lim_{n \to \infty}\|w_n\|_2^2 = \lim_{n \to \infty} \|{\varrho_n}_x\|_2^2$
and\\ $\|{\varrho_n}_x\|_2^2 \rightarrow \|\varrho_x\|_2^2$, 
since $\varrho_n \rightarrow \varrho$ in $X_A$.\\
Hence,$\| w_x\|_2^2 \leq \lim_{n \to \infty} \|{w_n}_x\|_2^2 = \|\varrho_x\|_2^2.$
But it is easy to see that 
$$\| w_x\|_2^2 = \int (|u_x|^2 + (|v_x|^2) dx \geq \int _{u^2 + v^2 > 0} \Big(\frac {(u u_x + v v_x)^2}{u^2 + v^2}\Big)^2 dx = \|\varrho_x\|_2^2,$$
because $ ({u u_x + v v_x}^2) \leq (u^2 + v^2) (|u_x|^2 + (|v_x|^2).$ Hence,$\|{w_n}_x\|_2^2 \rightarrow \| w_x\|_2^2$. We also have that 
$ w_n \rightharpoonup w$, weakly in $ {\tilde {X}}_A $. In particular, by compactness, $ w_n \rightarrow w$ in $(L^2(\R))^2 \cap (L^4(\R))^2.$\\
Since $\mathcal{\tilde {E}}(w_n) \rightarrow \mathcal{\tilde {I}}_c = \mathcal{\tilde {E}}(c)$, we derive that $\int x^2|w_n|^2dx \rightarrow  \int x^2|w|^2dx$ and so $\| w_n\|_{{\tilde {X}}_A}^2 \rightarrow \| w \|_{{\tilde {X}}_A}^2$. We conclude that $w_n \rightarrow w$ in $ {\tilde{X}}_A $, and this achieves the proof of Theorem 5.

\begin{Remark}
We would like to remark that in the semilinear case, namely when $\lambda=0$, we can simplify some arguments. Indeed,
by applying the Fourier transform to \eqref{(1.14)}, we can solve explicitly this equation and derive
\begin{equation} 
\label{(1.13bis)}  \rho =  \mathcal{F}^{-1} \Big(\frac{\mathcal{F}|u|^2}{b+4\pi^2\xi^2}\Big ).
\end{equation}
The energy functional is then given by:\\
\begin{equation} 
\begin{split}
\label{(1.15bis)} \mathcal{E} (u) = \frac {1}{2} \int |u_x|^2 dx + \frac {1}{2} H^2 \int x^2|u|^2 dx \\ - \frac {1}{4} \int |u|^4 dx 
- \frac {1}{4} \int \frac {|\mathcal{F}|u|^2|^2} {1+4\pi^2\xi^2} d\xi, \, u\in X_A.
\end{split}
\end {equation}
We can use directly \eqref{(1.13bis)} to obtain an estimate on $\rho$. To prove the symmetry of minimizers, we can use Proposition 3.2 in ~\cite{J}, noticing that $(|u|^2)^{\star} = |u^{\star}|^2$, to deduce that
\begin{equation}
\label{(5.10)} \int \frac {|\mathcal{F}|u|^2|^2} {1+4\pi^2\xi^2}  d\xi \leq \int \frac { |\mathcal{F}(|u|^2)^{\star}|^2}{1+4\pi^2\xi^2}  d\xi 
= \int \frac{ |\mathcal{F}(|u^{\star}|^2)|^2}{1+4\pi^2\xi^2}  d\xi .
\end{equation}

\end{Remark}

\section{Bifurcation structure}
This section is devoted to the study of the bifurcation structure of solution to the minimization problem \eqref{(1.16)} namely we prove Proposition \ref{propbif}. We begin by showing a Pohozaev identity which is also of independent interest.

\begin{Lemma}[Pohozaev identity]
\label{lempoho}
Let $u\in X_A$ be a solution to \eqref{(1.14)}. Then we have
$$2\| u_x\|_2^2 - 2H^2 \|xu\|_2^2 -\frac{1}{2} \|u\|_4^4 + \int u^2 x\rho_x (|u|^2) dx =0  .$$
\end{Lemma}

\begin{proof}
To simplify notation, we set $\rho:= \rho (|u|^2)$. Multiplying the equation \eqref{(1.14)} by $xu$ and integrating by parts, we get
$$\| u_x \|_2^2 - 3H^2 \|xu\|_2^2 + \dfrac{1}{2}\|u\|_{4}^4 + \int u^2 (\rho +x \rho_x) dx- \mu \|u\|_2^2=0 .$$
On the other hand, multiplying the equation by $u$ and integrating by parts, we get
\begin{equation}
\label{eqint}
\| u_x \|_2^2 +H^2 \|xu\|^2_2 - \int \rho u^2 dx - \|u\|_4^4  + \mu \|u\|_2^2 =0.
\end{equation}
So combining the two previous lines, we find
$$2\| u_x \|_2^2 - 2H^2 \|xu\|_2^2 -\frac{1}{2} \|u\|_4^4 + \int u^2 x\rho_x  dx=0.  $$
\end{proof}

Let us denote by $u_c$ a function achieving the minimum for the problem \eqref{(1.16)} and by $\mu_c$ its lagrange multiplier. We also set $\lambda_0$ for the first eigenvalue of the harmonic oscillator $-\partial_{xx}+H^2 x^2 $. We will show that $\mu_c$ converges to $-\lambda_0$ when the mass $c$ goes to $0$. 
\begin{Proposition}
  \label{prop3}
  We have
  $$\lim_{c\rightarrow 0} \mu_c =- \lambda_0.$$
\end{Proposition}

\begin{proof}
In a first time, we are going to show that $-\mu_c \leq \lambda_0$. Multiplying the equation satisfied by $u_c$ by $u_c$ and integrating by parts, we get
$$- c^2 \mu_c = \|(u_c)_x \|_2^2 +H^2\|xu_c\|_2^2 -\|u_c\|_4^4 - \int \rho (|u_c|^2) u_c^2 dx = 2E(u_c) - \dfrac{\|u_c\|_4^4}{2}.$$
Thus, we deduce that
$$-\mu_c \leq \dfrac{2 E(u_c)}{c^2}.$$
 Let $u_0$ be the eigenfunction associated to $\lambda_0$ namely $\|u_0\|_{X_A}^2 =\lambda_0$ and $\|u_0\|_2=1$. We set $v_c =cu_0$.
Using that $u_c$ is a minimiser of problem \eqref{(1.16)}, we have that $E(u_c)\leq E (v_c)$ and
$$ \dfrac{2 E(cu_0)}{c^2} =\|(u_0)_x \|_2^2 +H^2\|xu_0\|_2^2 - \dfrac{c^2}{2} \int u_0^4 dx - \int \rho (|u_0|^2) u_0^2 dx \leq \lambda_0.  $$
This proves that $-\mu_c \leq \lambda_0$.

Using Pohozaev's identity (see Lemma \ref{lempoho}) and \eqref{eqint}, we have
$$2 \|( u_c)_x\|_2^2 + \int u_c^2 (\frac{x\rho_x (|u_c|^2)}{2} - \rho (|u_c|^2)) dx - \frac{5}{4} \|u_c\|_4^4 + \mu_c \|u_c\|_2^2 =0. $$
So, recalling that $-\mu_c \leq \lambda_0$, we have for a constant $M>0$ not depending on $c$ that
$$\|(u_c)_x\|_2^2 \leq Mc^2 + M  \|u_c\|_4^4 - \int u_c^2 (\frac{x\rho_x (|u_c|^2)}{2} - \rho (|u_c|^2)) dx .$$
Notice that, integrating by parts and using radial coordinates,
\begin{align*}
 \int u_c^2 x\rho_x (|u_c|^2) dx& =\int u_c^2 (x\rho (|u_c|^2))_x dx -  \int u_c^2 \rho (|u_c|^2) dx\\
&= -2\int u_c (u_c)_r  r \rho (|u_c|^2) dr-  \int u_c^2 \rho (|u_c|^2) dx \\
&\geq - \int u_c^2 \rho (|u_c|^2) dx. 
\end{align*}
In the last inequality, we used that $u_r \leq 0$.
So, by \eqref{(5.2)}, we obtain, for some constant $M$ not depending on $c$,
$$\|(u_c)_x\|_2^2 \leq Mc^2 +M  \|u_c\|_4^4 +M \int u_c^2  \rho (|u_c|^2) dx \leq   Mc^2 +M  \|u_c\|_4^4.$$
The Gagliardo-Nirenberg's inequality \eqref{(5.3)} and Young's inequality then imply that
$$\|(u_c)_x\|_2^2 \leq M c^2.$$

We have, by definition of $\lambda_0$,
\begin{align*}
-\mu_c& = \dfrac{\|(u_c)_x\|_2^2 +H^2 \|x u_c\|_2^2}{c^2} - \dfrac{\|u_c\|_4^4 +\int \rho (|u_c|^2) u_c^2 dx}{c^2} \\
&\geq \lambda_0 - \dfrac{\|u_c\|_4^4 +\int \rho (|u_c|^2) u_c^2 dx}{c^2}.  
\end{align*}
Then, using \eqref{(5.2)} and Gagliardo-Nirenberg's inequality \eqref{(5.3)}, we deduce that, for some constant $k$ not depending on $c$,
$$-\mu_c \geq \lambda_0 -kc^4. $$
Taking $c\rightarrow 0$, the result follows.

\end{proof}

Adapting the proof of Proposition $6.7$ of \cite{Hirose}, we can prove the uniqueness of positive solution to \eqref{(1.14)} namely Proposition \ref{propuni}.

\begin{proof}[Proof of Proposition \ref{propuni}]
We denote by $u(r,\alpha_1)$ the radial solution to \eqref{(1.14)} such that $u(0,\alpha_1)=\alpha_1$. Suppose that there exists two number $0<\alpha_1<\tilde{\alpha}_1$ such that $u(r,\alpha_1)$ and $u(r,\tilde{\alpha}_1)$ are two positive radial solution decaying to $0$ at infinity. To simplify notation, we set $u(r)=u(r,\alpha_1)$ and $\eta (r)=u(r,\tilde{\alpha}_1)$. Let $\psi = \eta - u $. In the following, we denote by $u'= \partial_r u$. Then $\psi$ satisfies
\begin{equation}
\label{eqpsi}
\psi'' - (\lambda +r^2 )\psi + \dfrac{|\eta|^2 \eta - |u|^2 u +\rho (|\eta|^2) \eta - \rho (|u|^2) u}{\eta - u}\psi=0.
\end{equation}
Multiplying the previous equation by $u$ and multiplying \eqref{(1.14)} by $\psi$, taking the difference and integrating by parts, we find
\begin{align*}
\psi' (r) u(r) - u^\prime (r) \psi (r)&= \int_0^r (u^3 +\rho (|u|^2 )u) \psi dx\\
& - \int (\eta^3 + \rho (|\eta|^2) \eta -u^3 -\rho (|u|^2) u) u dx \\
&= \int (u^{2} +\rho (|u|^2) - \eta^{2} - \rho (|\eta|^2)) \eta u dx.
\end{align*}
Observe that the left-hand side goes to $0$ as $r\rightarrow \infty$ whereas if we assume that $\eta (r)>u(r)$ for all $r\geq 0$, the left-hand side converges to a negative constant. So there exists $\gamma_1$ such that $\eta (\gamma_1)=u(\gamma_1)$ (by the maximum principle, we can show that $\rho(|u|^2) -\rho (|\eta|^2)<0$).

Next, we will show that it is in fact the only intersection point between $u$ and $\eta$. Indeed, suppose by contradiction that there exists $\gamma_2>\gamma_1$ such that
$$0<\eta (r) < u(r) \ \text{for}\ r\in (\gamma_1 ,\gamma_2),\ u(\gamma_2)=\eta (\gamma_2).$$
This implies that
$$\psi (r)<0\ \text{for}\ r\in (\gamma_1 ,\gamma_2),\ \psi^\prime (\gamma_1) <0,\ \psi^\prime (\gamma_2)>0\ \text{and}\ \psi(\gamma_1)=\psi (\gamma_2).$$
Let $\xi$ be a solution to
\begin{equation}
\label{eqxi}
\begin{cases}
&\xi'' - (\lambda +r^2 ) \xi +[p |u|^{p-1}+  \partial_u (\rho (|u|^2) u )]\xi = 2r u,\ r>0\\
& \xi (0)=0,\ \xi^\prime (0)= (\lambda -\rho (\alpha^2) ) \alpha -\alpha^p.
\end{cases}
\end{equation}

In fact, we can think of $\xi$ as $u^\prime$ noticing that $(\rho (|u|^2) u )_x= u^\prime (\rho (|u|^2) + u  \partial_u (\rho (|u|^2)))$. Let
$$\chi (r)=p|u|^{p-1} +\partial_u (\rho(|u|^2)u) - \dfrac{|\eta|^{p-1} \eta - |u|^{p-1} u+ \rho (|\eta|^2) \eta -\rho (|u|^2) u}{\eta -u}.$$
Observe that the function $u\rightarrow \rho (|u|^2) u$ is convex. Indeed $\partial_{uu}(\rho (|u|^2) u)= u\partial_{uu} \rho (|u|^2) + 2 \partial_u \rho (|u|^2)$ where $\partial_u \rho (|u|^2) :=f$ is the solution to
$$-f'' - 3 \lambda ((\rho (|u|^2)^\prime)^2 f_x)' +bf = 2u, $$
and, $\partial_{uu}\rho (|u|^2)=g$ is the solution to
$$-g'' - 3\lambda  ((\rho (|u|^2)^\prime)^2 g_x)' +bg = 2 + 6\lambda ((\partial_u \rho (|u|^2))^2)^{\prime \prime}.  $$
By the maximum principle, we see that $f\geq 0$ and $g\geq 0$ (since by comparison principle we can show that $\rho (t x_1)\leq t \rho (x_1)$ which implies, using once more comparison principle that $\rho (t x_1 +(1-t) x_2)\leq t \rho (x_1) + (1-t)\rho (x_2)$, for all $x_1,x_2\geq 0$ and $t\in [0,1]$). Using this and the convexity of $u^p$, we see that $\chi (r)>0$ when $r \in (\gamma_1 ,\gamma_2)$. Taking the difference of \eqref{eqpsi} multiplied by $\xi$ and \eqref{eqxi} multiplied by $\psi$ and integrating by parts on $[\gamma_1 ,r]$, we find
\begin{equation}
\label{eqwronk}
\xi (r) \psi^\prime (r) - \xi^\prime (r) \psi (r) = \xi (\gamma_1) \psi^\prime (\gamma_1) +\int_{\gamma_1}^r (\chi (s) \xi (s) \psi (s) -2 s  u(s) \psi (s))ds.
\end{equation}
Taking $r=\gamma_2$ in the previous identity, we get
$$\xi (r_2) \psi^\prime (\gamma_2) = \xi (\gamma_1) \psi^\prime (\gamma_1) + \int_{\gamma_1}^{\gamma_2} [\chi (s) \xi (s) \psi (s) -2 s  u(s) \psi (s) ]ds.$$
This is a contradiction since the left-hand side is strictly negative while the right-hand side is strictly positive. This establishes that $u$ and $\eta$ intersect exactly once.

Finally, we show that $\eta$ has to change sign. Suppose by contradiction that 
$$0<\eta (r) < u(r)\ \text{for}\ r\in (\gamma_1 ,\infty).$$
This implies that $\psi (r)<0$ for $r\in (\gamma_1 ,\infty)$, $\psi^\prime (\gamma_1)<0$ and $\psi (\gamma_1)=0$. Since $u,u^\prime$ and $u''$ go to $0$ as $r\rightarrow \infty$ (and the same for $\eta$), we see that the right-hand side of \eqref{eqwronk} goes to $0$ taking $r\rightarrow \infty$ whereas the righ-hand side converges to a positive constant. Therefore, $\eta$ cannot be positive everywhere and consequently $u$ is the unique positive radial solution to our equation.
\end{proof}

We are finally in position to prove our bifurcation result, i.e. Proposition~\ref{propbif}.
\begin{proof}[Proof of Proposition \ref{propbif}]
Since $\lambda_0$ is a simple eigenvalue, we can apply standard bifurcation results (see for instance \cite[Theorem 2.1]{H2}) to deduce that $(\lambda_0 ,0)$ is indeed a bifurcation point and that the branch is unique provided that we are sufficiently close to the bifurcation point. Next, the previous Lemma guarantees that the minimizer of \eqref{(1.16)} $u_c$ actually belongs to this branch at least for $c>0$ small enough. Finally, we use our uniqueness result Proposition \ref{propuni} to see that the set $\{u_c ,c>0\}$ is convex and therefore included in the bifurcating branch.  
\end{proof}

\section{Numerical simulations} \label{Simulations}
In this section we perform some numerical simulations to illustrate our results.
We investigate the limit $\mu \to -\lambda_0$ mentioned in the previous section, and
analyse the behaviour of standing waves with the variation of the intensity of the magnetic field 
$\mathbf H$.

\subsection{Numerical method}
\label{sec:numerical-method}
Our first goal is to numerically approximate the standing waves \eqref{(1.11)}, according to the system \eqref{(1.12)}.
Following \cite{HS et al}, we use a shooting method. However, in the present case, the director field angle $\rho = \rho(|u|^2)$ acts as an additional potential type term, depending on $u$ itself. Due to this, we perform a Picard iteration and look for a fixed point $u$ of the operator $\varphi\mapsto \Phi(\varphi)$, where $\Phi(\varphi)$ is the solution of 
\begin{equation}
  \label{eq:2}
  \aligned
  u_{xx} - H^2 x^2 u + |u|^2 u + \rho(\varphi) u = \mu u,
  \endaligned
\end{equation}
with $\rho(\varphi)$ solving
\begin{equation}
  \label{eq:3}
  \aligned
  -\rho_{xx} - \lambda(\rho^3_x)_x + b \rho = \varphi
  \endaligned
\end{equation}
and with boundary conditions $u(0)  = u_0>0$, $\rho(0) = \rho_0 >0$, $u(\infty) =u'(0)  = \rho(\infty)= \rho'(0) = 0$.

According to the results in previous sections, we look for $u\in\R$ even, smooth, vanishing at infinity, strictly positive and decreasing with $|x|$. For convenience, we shall denote the class of functions verifying these conditions by $\mathcal{V}$. Although there is no result giving a similar structure for $\rho(x)$, it is natural to assume that $\rho$ satisfies the same hypotheses as $u$, at least for small $\lambda$, and so we look for $u,\rho \in \mathcal{V}$.

We now describe our procedure in more detail. First, equations \eqref{eq:2},\eqref{eq:3} can be recast as a first-order system:
\begin{equation}
  \label{eq:4}
  \left\{\aligned
    &u_x = w
    \\
    &w_x =  H^2 x^2 u - |u|^2 u - \rho(\varphi) u + \mu u
    \\
    &\rho_x = v
    \\
    &v_x =  -\lambda(v^3)_x + b \rho - \varphi^2,
  \endaligned\right.
\end{equation}
with boundary conditions $u(0) = u_0>0$, $\rho(0)=\rho_0>0$, $v(0)=w(0)=0$, and $\varphi \in \mathcal V$.

At each stage in the Picard iteration, we need, for a given $\varphi\in \mathcal V$, to find $(u,w,\rho,v)$ solving \eqref{eq:4}. As mentioned, we employ a shooting method, which we now describe. Suppose that we have computed $\rho(\varphi), v(\varphi)$, and wish to compute $u,w$. The idea is to adjust the initial value $u(0)=u_0$ so that $u(\infty) =0$. Following \cite{HS et al}, $u_0$ should verify $u_0 = \sup\{ \beta >0 : u(x;\beta) >0, x>0\}$ where $u(x;\beta)$ is the solution of \eqref{eq:2} with $u(0) =\beta$, $u\in\mathcal V$. At each step of the shooting method, we look for $u_0$ in an interval $[a_n,b_n]$. We set $u_{0,n} = (a_n+b_n)/2$ and solve the first two equations of \eqref{eq:4} using an explicit Euler scheme (which is sufficient for our purposes) with $w(0) = 0$. Then, if $u$ attains negative values for some $x$, we set $a_{n+1} = u_{0,n}$, $b_{n+1} = b_n$, thus decreasing $u_{0,n+1}$. Conversely, if $u(x)$ is increasing at some point (so that it does not belong in the class $\mathcal V$), we set $a_{n+1}=a_n$ and $b_{n+1} = u_{0,n}$, which increases $u_{0,n+1}$.

The procedure to compute $\rho$ and $v$ is similar, except that the behaviour of $\rho$ exhibits an inverse dependence on the initial value $\rho(0)$; thus in each iteration of the shooting method the value of $\rho(0)$ is increased when $\rho$ becomes negative, and decreased when $\rho$ becomes increasing.

Let us mention that on each iteration of the shooting method, the equation for $v$ in \eqref{eq:4} contains a nonlinear term when $\lambda \neq 0$. The discretized equation reads
\begin{equation}
  \aligned
  \frac{v_{j+1} -v_j}{dx} = -\lambda \frac{(v_{j+1})^3 - (v_j)^3}{dx} + b\rho_j - (\varphi_j)^2,
  \endaligned
\end{equation}
and so we use a Newton method at each step to approximately solve for $v_{j+1}$.

As a starting point to the Picard iteration, we take $u^{(0)}(x) \in \mathcal V$ as the solution with $\rho=0$, that is, $u^{(0)}$ solves $u_{xx} - H^2 x^2 u + |u|^2 u  = \mu u$ with $u^{(0)}\in \mathcal V$.  
With an initial guess $u^{(0)}(x) \in \mathcal V$ for the Picard iteration in hand, we compute $u^{(1)}(x), w^{(1)}(x)$, and so on, using the shooting method, according to
\begin{equation}
  \label{eq:5}
  \left\{\aligned
    &u^{(n)}_x = w^{(n)}
    \\
    &w^{(n)}_x =  H^2 x^2 u^{(n)} -  (u^{(n)})^3 - \rho(u^{(n-1)}) u^{(n)} + \mu u^{(n)},
  \endaligned\right.
\end{equation}
where $\rho(u^{(n-1)})$ solves
\begin{equation*}
  \left\{\aligned
  &\rho_x = v,
  \\
  &v_x =  -\lambda(v^3)_x + b \rho - (u^{(n-1)})^2
  \endaligned\right.
\end{equation*}
(also using the shooting method),
with boundary conditions $u(0) = u_0>0$, $\rho(0)=\rho_0>0$, $v(0)=w(0)=0$.

\subsection{Numerical results}
\label{sec:numerical-results}

\providecommand{\mymy}{1}
\begin{figure}[!h] \includegraphics[width=\mymy\textwidth]{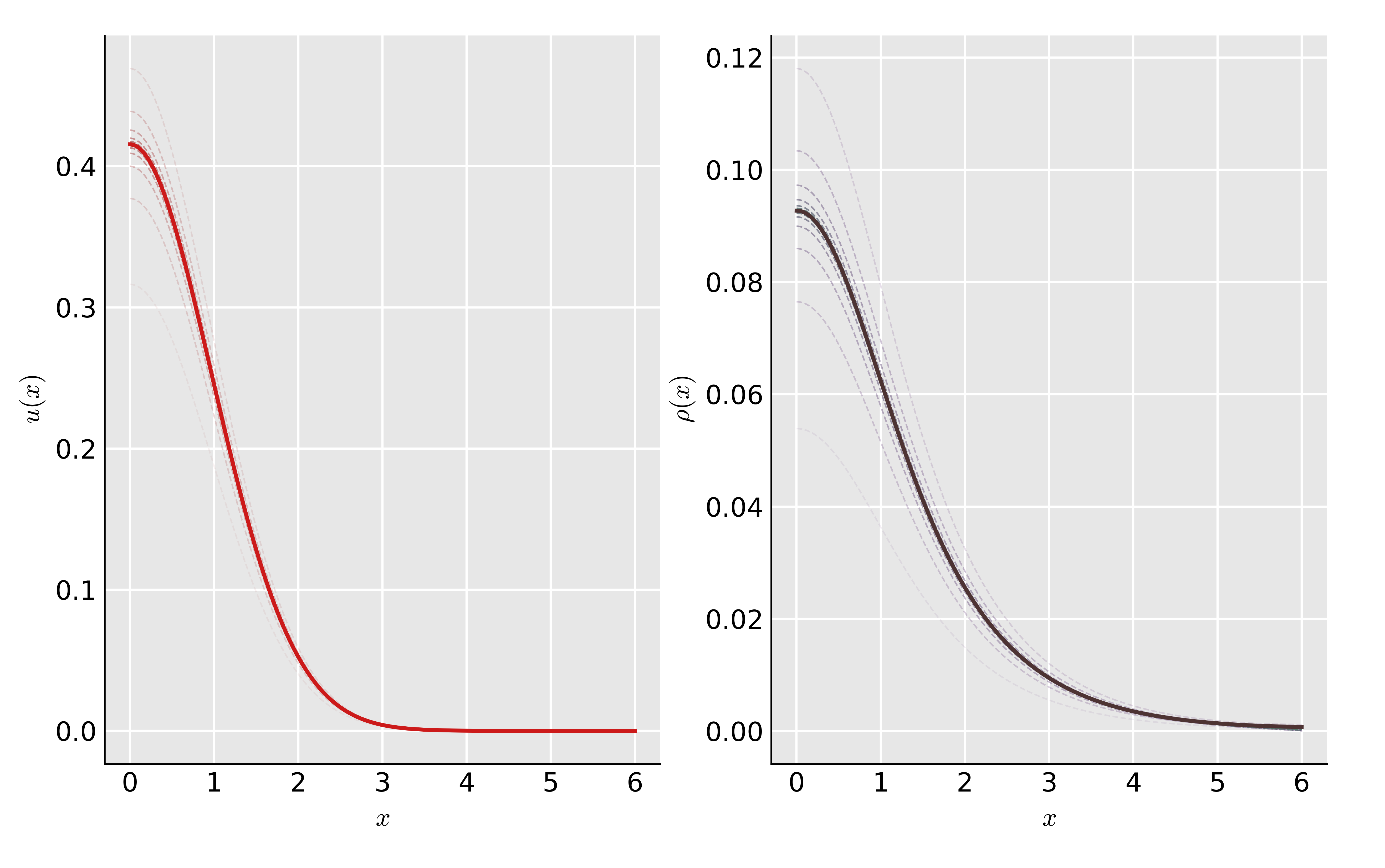}
 \caption{Numerical approximation of the standing wave $u(x)$ and the director field angle $\rho(x)$, solutions to \eqref{(1.12)}, computed using a shooting method and Picard iteration (Picard iterations in dashed lines). Parameters are $H=1,\mu=-0.8,\lambda=0.1, b=1$.}
 \label{fig:1}
\end{figure}

In Fig.~\ref{fig:1}, we plot the standing wave $u(x)$ and the director field angle $\rho(x)$ calculated according to the procedure described previously. The dashed lines correspond to the iterations of the Picard method. For this simulation, we have used a spatial step $dx = 0.002$ (corresponding to $3000$ spatial points) and 15 Picard iterations.

Next, we illustrate the result of Proposition~\ref{prop3}. First, note that it is easy to see that $u^*(x) = e^{-\frac{H}2x^2}$ is the first eigenfunction of the harmonic oscillator $-\partial_{xx} + H^2 x^2$, with eigenvalue $\lambda_0=H$. Note that in our notations, the parameter $-\mu$ plays the role of $\lambda_0$.
In parallel to \cite{HS et al}, and in accordance with Proposition~\ref{prop3}, we verify numerically that the $L^2$ norm of $u_\mu$ goes to zero as $\mu \to -\lambda_0^+$. Taking $H=\lambda_0=2,$ we show in Fig.~\ref{fig:2} the numerical solutions of \eqref{(1.12)} for various values of $\mu \to -\lambda_0^+$. We can see that the solutions appear to converge to zero, although the convergence is very slow. In Fig.~\ref{fig:3}, we show how the $L^2$ norm of $u=u_\mu$ varies as the Lagrange multiplier $\mu$ tends to the value $-\lambda_0$. Our numerical tests indicate that, although slow, the convergence to zero of the $L^2$ norm of $u_\mu$ is verified, in accordance with Proposition~\ref{prop3}.

\providecommand{\mymy}{1}
\begin{figure}[!h] \includegraphics[width=\mymy\textwidth]{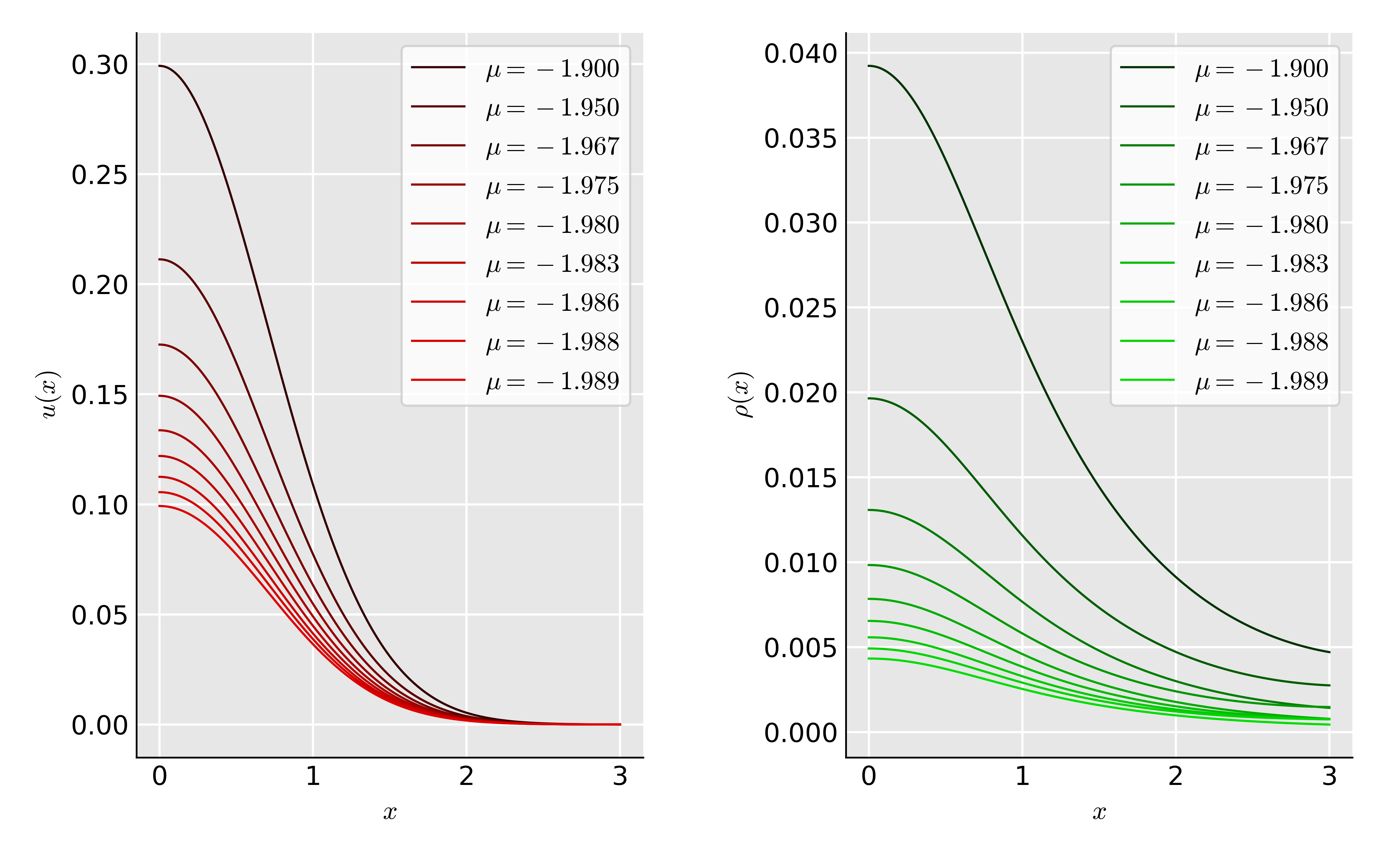}
  \caption{Numerical approximation of the standing wave $u(x)$ and the director field angle $\rho(x)$, solutions to \eqref{(1.12)}, with $\mu\to -\lambda_0$. Parameters are $H=\lambda_0 =2,\lambda=0.1, b=1$.}
  \label{fig:2}
\end{figure}

\providecommand{\mymy}{0.2}
\begin{figure}[!h]
  $$
  \includegraphics[width=0.5\textwidth]{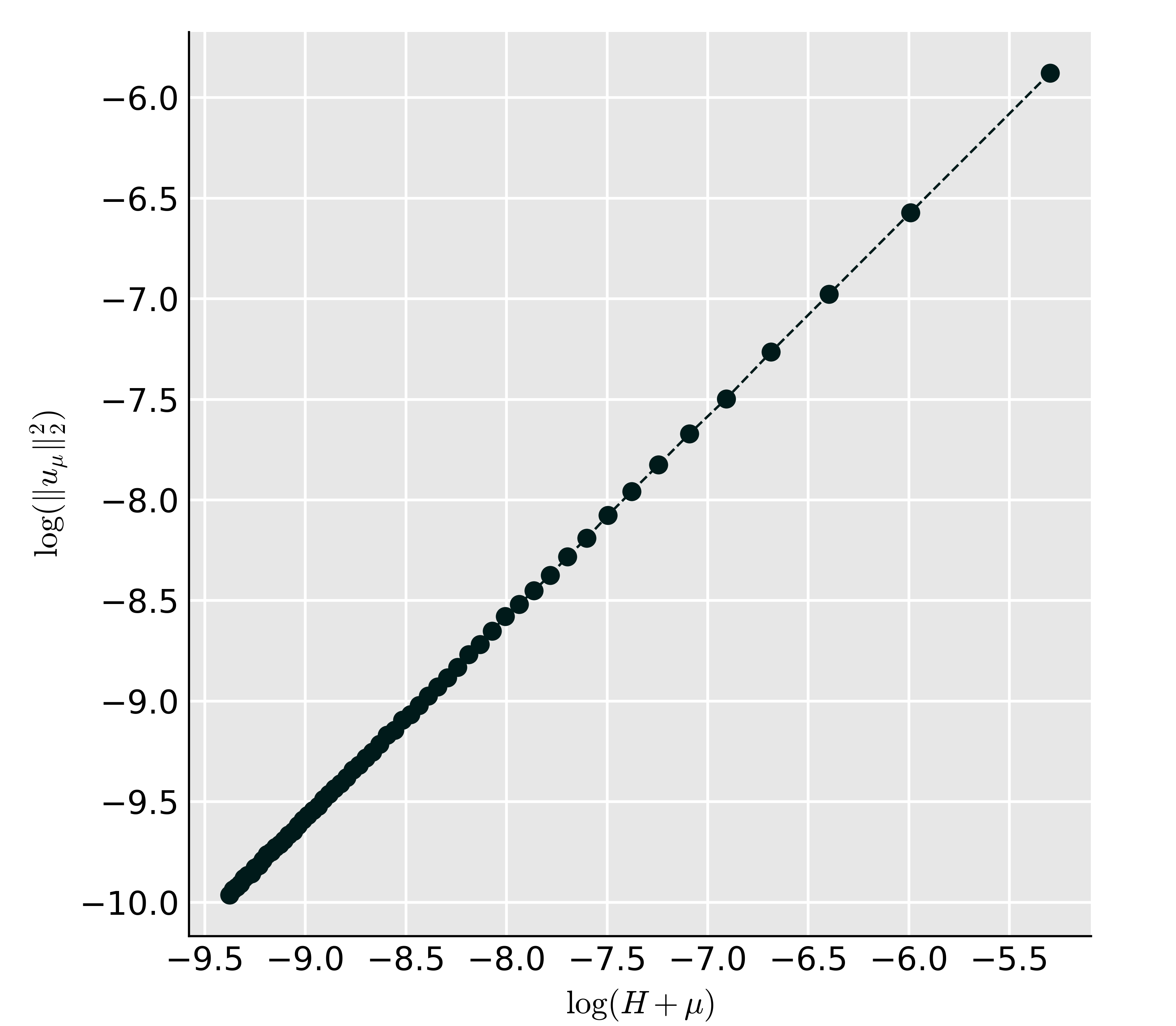}
  \includegraphics[width=0.5\textwidth]{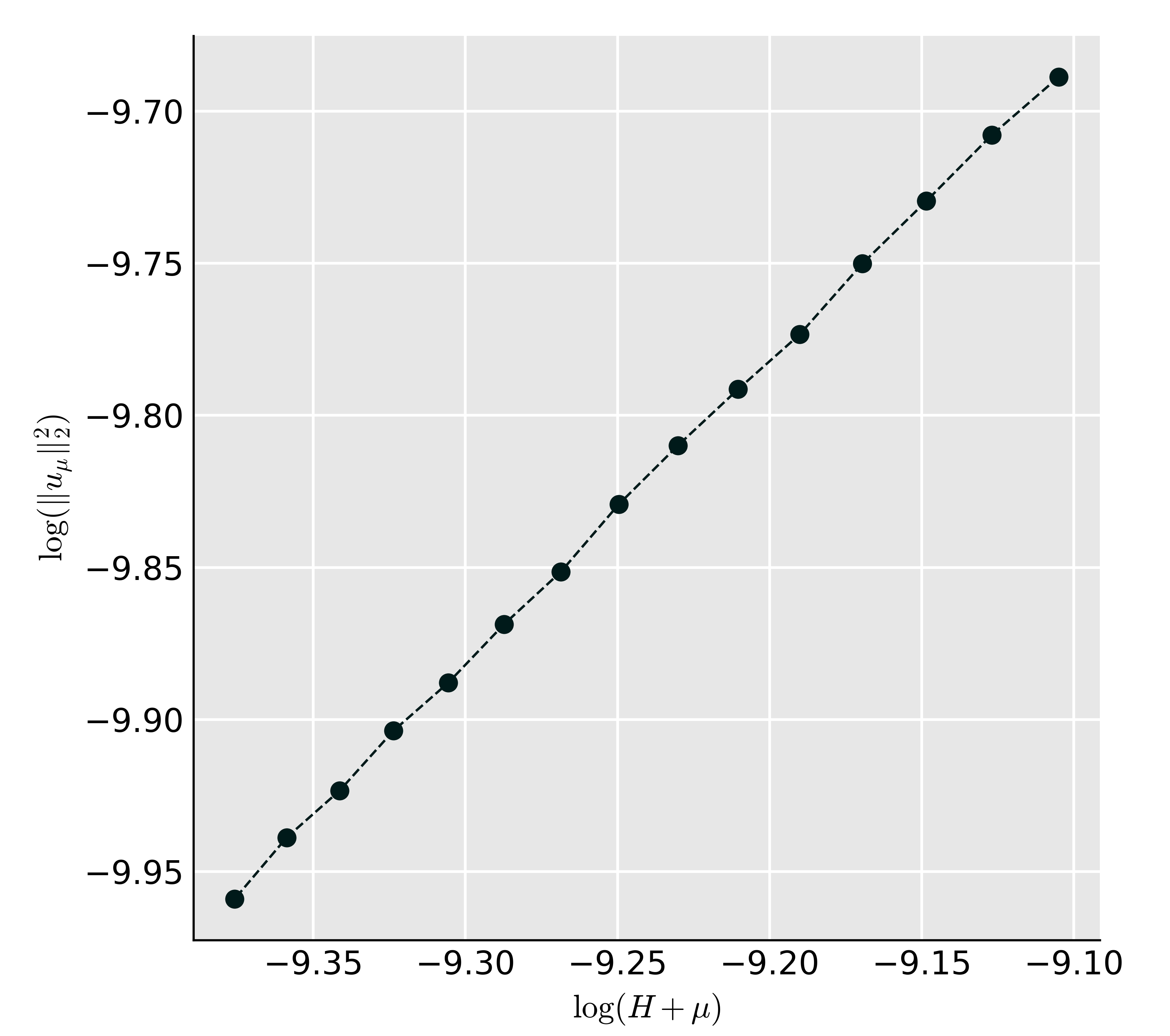}
  $$
 \caption{The norm $\|u\|_2^2$ as a function on the Lagrange multiplier $\mu$ as $\mu\to -\lambda_0 = - H =-2$, in $\log$-$\log$ scale. The values of $\mu$ are the same as in Fig.~\ref{fig:2}, but $\mu$ is ranging from $-1.9$ to $-1.9999153$, taking 60 values (left). On the right is a zoom on the last 15 values of $\mu$.}
 \label{fig:3}
\end{figure}

Next, we investigate numerically the behaviour of the standing wave when the intensity of the magnetic field, $H$, is varied. It turns out that for each set of parameters that we analyzed, there is a maximum (relatively small) value of $H$ such that our numerical method diverges for larger values of $H$. This may be related to the observation that the behaviour of $u$ (and $\rho$) with respect to $u(0)$ is very sensitive to perturbations: any arbitrarily small perturbation of the $u(0)$ found by the shooting method produces a solution which (numerically at least) quickly blows up exponentially. The desired solution appears to be unstable in this sense, and this effect appears more markedly for larger values of $H$. Still, in Fig.~\eqref{fig:4} we show the behaviour of the solution for $H$ between 0 and 2, which lets us nevertheless see the general trend. In particular, it is clear that the director field angle becomes more concentrated at the origin for larger values of $H$.

\providecommand{\mymy}{1.}
\begin{figure}[!h] \includegraphics[width=\mymy\textwidth]{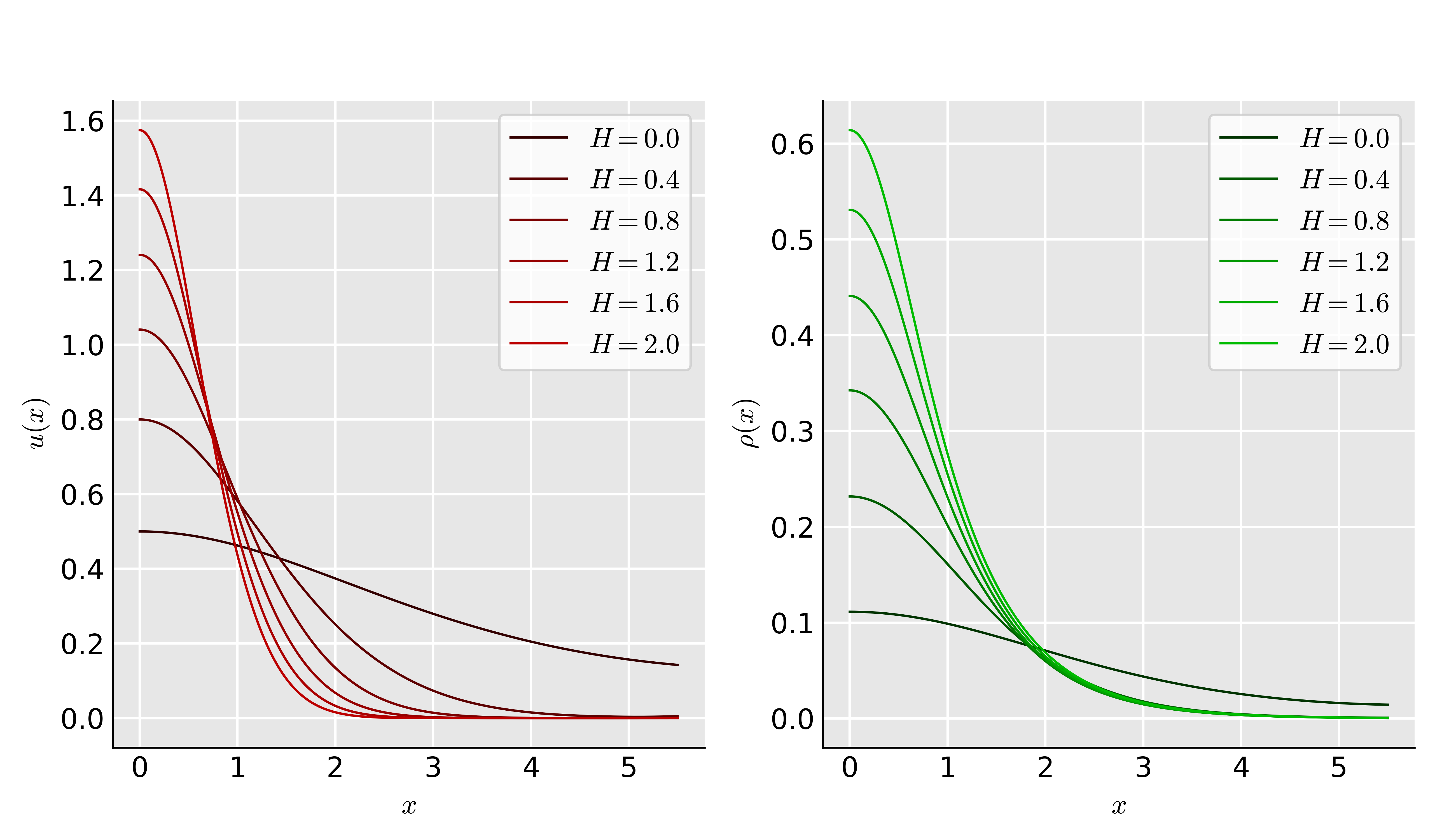}
  \caption{Numerical approximation of the standing wave $u(x)$ and the director field angle $\rho(x)$, solutions to \eqref{(1.12)}, with varying magnetic field intensity $H$. Parameters are $\mu =0.2$, $\lambda=0.1,$ $b=2$.}
  \label{fig:4}
\end{figure}

\subsection*{Acknowledgements} 
 P. Amorim was partially supported by Conselho Nacional de Pesquisa (CNPq) grant No. 308101/2019-7. J.-B. Casteras and J.P. Dias  were supported by the Funda\c c\~{a}o para a Ci\^{e}ncia e a Tecnologia (FCT) through the grant UIDB/04561/2020. J.P.Dias is indebted to R.Carles, H.Frid, E.Du\-cla Soares, A.Farinha Martins and L.Sanchez for important suggestions and comments. 

\end{document}